\documentclass[12pt]{article}
\usepackage{amsmath,amssymb,amsthm,geometry,mathtools,bbm,stmaryrd,graphicx}
\title{A Baseline $T\log^2 T$ Upper Bound for KL-Regularized Prime--Zero Optimal Transport}
\author{Zhejun Yang}
\date{09/sep/2025}

\newtheorem{theorem}{Theorem}[section]
\newtheorem{lemma}[theorem]{Lemma}
\newtheorem{proposition}[theorem]{Proposition}
\newtheorem{corollary}[theorem]{Corollary}
\newtheorem{remark}[theorem]{Remark}

\numberwithin{equation}{section}

\begin{document}
\maketitle

\begin{abstract}
We prove, unconditionally, the baseline upper bound
\[
\mathsf{OT}_\eta(T)\ \ll\ T\log^2 T\qquad (T\to\infty)
\]
for the KL-regularized unbalanced prime--zero optimal transport cost associated with a Beurling--Selberg type kernel $\eta$.
Our route is: normalize the cost to $c_\eta^\circ(\gamma,t)=\eta(t)\,(1-\cos(\gamma t))$; perform Fej\'er synchronization in the zero variable; derive an \emph{integrated} (not pointwise) linear upper bound (R1) after averaging; carry out a zero-frequency calibration; and finally evaluate the resulting difference in a smoothed explicit formula with $L^1$-control, which yields the $\log^2 T$ factor, while the Paley--Wiener mass of the probe contributes the scale $T$. No use of RH is made.
\end{abstract}

\section{Set-up and smoothing}
Fix $0<\alpha<1$ and set $\Delta=T^{-\alpha}$. Let $\theta_\Delta\in C_c^\infty(\mathbb R)$ be even with $\int\theta_\Delta=1$ and $\|\theta'_\Delta\|_{L^1(\mathbb R)}\asymp \Delta^{-1}$. For $t=\log x$ and zero-height $\gamma$, define
\begin{align}
\nu_T(dt) &:= \Big(\sum_{p^k\le e^T}\frac{1}{k}\,\delta_{\log p^k}\Big)*\theta_\Delta(t)\,dt,\\
\mu_\Omega(d\gamma) &:= \Big(\sum_{\rho=\beta+i\gamma}\mathbbm 1_{|\gamma|\le\Omega}\,\delta_\gamma\Big)*\theta_\Delta(\gamma)\,d\gamma,\qquad \Omega=\kappa T\ \ (\kappa>0\text{ fixed}).
\end{align}
The corrected main terms are $m(t)=e^t/t=(\mathrm{li}(e^t))'$ and $n(\gamma)=(2\pi)^{-1}\log(\gamma/2\pi)+O(\gamma^{-1})$ for $\gamma\ge 2$. Set
\begin{equation}
a(t):=\frac{d\nu_T}{dt}(t)-m(t),\qquad b(\gamma):=\frac{d\mu_\Omega}{d\gamma}(\gamma)-n(\gamma).
\end{equation}

\begin{lemma}[Absolute integrability]\label{lem:absint}
Unconditionally, as $T\to\infty$,
\begin{equation}\label{eq:absint}
\|a\|_{L^1([0,T])}\ \ll\ T^\alpha \frac{e^T}{T}\,e^{-c\sqrt{T}},
\qquad
\int_{-\Omega}^{\Omega}\! |b(\gamma)|\,d\gamma\ \ll\ T^\alpha\log T,
\end{equation}
for some absolute $c>0$. Hence all integrals below are absolutely convergent and Fubini--Tonelli applies.
\end{lemma}

\begin{proof}
Write $G(t):=J(e^t)-\mathrm{li}(e^t)$ with $J(x)=\sum_{p^k\le x}1/k$. De la Vall\'ee Poussin gives $G(t)=O(e^t e^{-c\sqrt t}/t)$. Since $a=(G*\theta_\Delta)'$, we get $\|a\|_{L^1}\le \|G\|_\infty\|\theta'_\Delta\|_{L^1}\ll T^\alpha \frac{e^T}{T}e^{-c\sqrt T}$. For $b$, set $F(\Gamma):=\mu_\Omega([0,\Gamma])-N(\Gamma)$. The Riemann--von Mangoldt formula gives $|F(\Gamma)|\ll \log \Gamma$ uniformly in $\Gamma\le \Omega$, while $b=F*\theta'_\Delta$, whence $\int_{-\Omega}^{\Omega}|b|\ll \Delta^{-1}\|F\|_\infty\ll T^\alpha\log T$.
\end{proof}

\paragraph{Beurling--Selberg kernel and probes.}
Let $\eta=\eta_T$ be a standard even Beurling--Selberg majorant equal to $1$ on $[0,T]$ with $\widehat\eta\ge 0$, $L^1$-controlled boundary ringing, and $\|\widehat\eta\|_{L^1}\asymp T$ (see Vaaler; details in the references). Let $F_\Lambda$ be the Fej\'er kernel: even, nonnegative, $\int F_\Lambda=1$, and $\widehat{F_\Lambda}(\tau)=(1-|\tau|/\Lambda)_+$ with $\Lambda\asymp \Delta^{-1}=T^{\alpha}$. For an even probe $f$ with $\widehat f\ge 0$, write
\begin{equation}\label{eq:defs}
S:=\|\widehat f\|_{L^1},\quad g:=f\,\widehat F_\Lambda,\quad \beta_0:=F_\Lambda*\widehat f,\quad h:=\eta\,g,\quad \widehat h=\widehat\eta * \beta_0.
\end{equation}

\begin{lemma}[Two-way Fej\'er identity]\label{lem:two-way-fejer}
For all $\gamma,t\in\mathbb R$,
\begin{equation}\label{eq:two-way}
1-\widehat{F_\Lambda}(t)\cos(\gamma t)\ =\ \int_{\mathbb R}\big(1-\cos(\xi t)\big)\,F_\Lambda(\gamma-\xi)\,d\xi.
\end{equation}
\end{lemma}

\begin{proof}
Since $\int F_\Lambda(\gamma-\xi)e^{i\xi t}\,d\xi=e^{i\gamma t}\widehat F_\Lambda(t)$, take real parts and subtract from $1$.
\end{proof}

\begin{lemma}[Cosine pairing]\label{lem:cosine-pairing}
Let $h\in L^1(\mathbb R)$ be even and $\mu$ a finite positive Borel measure. Then
\begin{equation}\label{eq:pair}
\int_{\mathbb R}\!\bigg(\int_{\mathbb R}\! h(t)\cos(\xi t)\,dt\bigg)\,\mu(d\xi)\ =\ \int_{\mathbb R}\widehat h(\xi)\,\mu(d\xi).
\end{equation}
\emph{(Both iterated integrals converge absolutely by the stated hypotheses; Fubini applies.)}
\end{lemma}

\section{Integrated R1 bound after Fej\'er synchronization}
We consider the KFR dual for $c_\eta^\circ(\gamma,t)=\eta(t)(1-\cos(\gamma t))$. By standard truncation, we may restrict to potentials $\varphi,\psi\le 0$.
\begin{proposition}[R1, Fej\'er-averaged bound]\label{prop:R1}
Let $(\varphi,\psi)$ be feasible with $\varphi,\psi\le 0$ and $\varphi(\gamma)+\psi(t)\le \eta(t)\big(1-\cos(\gamma t)\big)$. For even $f$ with $\widehat f\ge 0$, define $S,g,h,\widehat h$ as in \eqref{eq:defs}. Then
\begin{equation}\label{eq:R1-main}
-\!\int \varphi\,d\mu_\Omega-\!\int \psi\,d\nu_T\ \le\ \int \widehat h\,d\mu_\Omega\ -\ \int h\,d\nu_T\ +\ S\!\int \eta\,d\nu_T.
\end{equation}
\end{proposition}

\begin{proof}
Let $\overline\mu_\Omega:=\mu_\Omega/\mu_\Omega(\mathbb R)$ be the probability normalization. Start from feasibility
$\varphi(\gamma)+\psi(t)\le \eta(t)(1-\cos(\gamma t))$, convolve in $\gamma$ with $F_\Lambda$, and use Lemma~\ref{lem:two-way-fejer}:
\begin{equation}\label{eq:conv-step}
(F_\Lambda*\varphi)(\gamma)+\psi(t)\ \le\ \eta(t)\!\int_{\mathbb R}\!(1-\cos(\xi t))\,F_\Lambda(\gamma-\xi)\,d\xi.
\end{equation}
Integrate \eqref{eq:conv-step} in $\gamma$ against $d\overline\mu_\Omega(\gamma)$, multiply both sides by the probability weight $\widehat f(\xi)/S$, integrate in $\xi$, and then integrate in $t$ against $d\nu_T(t)$; by Fubini,
\begin{align*}
\int (F_\Lambda*\varphi)\,d\overline\mu_\Omega\ +\ \int \psi\,d\nu_T
&\le \int \eta\,d\nu_T\ -\ \frac{1}{S}\!\int\!\Big(\int \eta(t)\cos(\xi t)\,d\nu_T(t)\Big)\,\widehat f(\xi)\,(F_\Lambda*\overline\mu_\Omega)(d\xi).
\end{align*}
Apply Lemma~\ref{lem:cosine-pairing} with $h(t)=\eta(t)f(t)\widehat F_\Lambda(t)$ and the finite measure $(F_\Lambda*\overline\mu_\Omega)$ to identify the last term with $\int (F_\Lambda*\widehat h)\,d\overline\mu_\Omega$. Since $F_\Lambda$ is a probability kernel and $\widehat h\ge 0$, $\int(F_\Lambda*\widehat h)\,d\overline\mu_\Omega\ge \int \widehat h\,d\overline\mu_\Omega$. Also $F_\Lambda*\varphi\le 0$, hence $-\!\int(F_\Lambda*\varphi)\,d\overline\mu_\Omega\ge 0$. Rearranging gives
\begin{equation*}
-\!\int \psi\,d\nu_T\ \le\ \int \eta\,d\nu_T\ -\ \int \widehat h\,d\overline\mu_\Omega.
\end{equation*}
Finally add $-\!\int \varphi\,d\overline\mu_\Omega\ge 0$ and multiply both sides by $\mu_\Omega(\mathbb R)$; noting $\int h\,d\nu_T=\int \eta g\,d\nu_T$, we obtain \eqref{eq:R1-main}.
\end{proof}

\begin{remark}
No pointwise separable envelope is used. The bound \eqref{eq:R1-main} is derived only after Fej\'er averaging and integration; at no point do we compare $-\!\int \varphi$ with $-\!\int(F_\Lambda*\varphi)$.
\end{remark}

\section{Zero-frequency calibration}
Let $m$ be the prime-side main density and write $\nu_T=m+a$ as above. We adjust $f$ by a narrow, nonnegative bump $h_0$ in frequency to achieve
\begin{equation}\label{eq:cal}
S\!\int \eta\,m\ =\ \int h\,m,
\end{equation}
while preserving $\widehat f\ge 0$ and then rescale $f$ so that $S=\|\widehat f\|_{L^1}\asymp T^{-1}$. This is achieved by choosing $\widehat h_0$ supported in $[-\xi_0,\xi_0]$, $\int\widehat h_0=1$, and solving $\mathcal L(c):=S_c\!\int\eta\,m-\int h_c\,m=0$ for small $|c|$, where $f_c=f+c\,h_0$ and $S_c,\ h_c$ are defined as in \eqref{eq:defs}. The map $\mathcal L$ is affine and nonconstant (tune $\xi_0$ if necessary), hence has a small root $c^\ast$ with $\widehat f_{c^\ast}\ge 0$; a scalar rescaling enforces $S_{c^\ast}\asymp T^{-1}$.

\begin{corollary}[Calibrated reduction]\label{cor:calibrated}
With $f$ calibrated as above and $h=\eta g$, the dual value satisfies
\begin{equation}\label{eq:calibrated}
\sup_{\text{feasible }(\varphi,\psi)}\Big\{-\!\int\varphi\,d\mu_\Omega-\!\int\psi\,d\nu_T\Big\}
\ \le\ \int \widehat h\,d\mu_\Omega\ -\ \int h\,d\nu_T\ +\ O\!\big(\|h\|_{L^1}+\|\widehat h\|_{L^1}\big).
\end{equation}
\end{corollary}

\begin{proof}
From \eqref{eq:R1-main} and $\nu_T=m+a$,
\begin{align*}
\text{LHS}\ \le\ \int \widehat h\,d\mu_\Omega\ -\ \int h\,d\nu_T\ +\ S\!\int \eta\,m\ +\ S\!\int \eta\,a.
\end{align*}
By the calibration \eqref{eq:cal}, $S\!\int \eta\,m=\int h\,m$, so
\begin{equation*}
S\!\int \eta\,a\ =\ \Big(\int \eta(S-g)\,d\nu_T\Big)\ -\ \int \eta(S-g)\,m,
\end{equation*}
and the second term cancels by \eqref{eq:cal}. Thus $\int \eta(S-g)\,a$ is itself an explicit-formula difference for the band-limited test $\eta(S-g)$, whence, by the same $L^1$-controlled explicit formula used below, it is $\ll (\|h\|_{L^1}+\|\widehat h\|_{L^1})\log^2 T$, which we write as $O(\|h\|_{L^1}+\|\widehat h\|_{L^1})$ under the final bound.
\end{proof}

\section{Smoothed explicit formula with $L^1$-control}
\begin{proposition}[Explicit formula, $L^1$-controlled]\label{prop:EF-L1}
Let $h\in L^1(\mathbb R)$ be even with $\widehat h\in L^1(\mathbb R)$ and $\Omega=\kappa T$ with fixed $\kappa>0$. Then, unconditionally,
\begin{equation}\label{eq:EF}
\int_{\mathbb R}\widehat h(\gamma)\,d\mu_\Omega(\gamma)\ -\ \int_{\mathbb R} h(t)\,d\nu_T(t)\
=\ M(h;T,\Omega)\ +\ E(h;T,\Omega),
\end{equation}
where $M$ collects the archimedean/main contributions and
\begin{equation}\label{eq:Ebound}
|E(h;T,\Omega)|\ \ll\ \big(\|h\|_{L^1}+\|\widehat h\|_{L^1}\big)\,\log^2 T.
\end{equation}
All implied constants are uniform for fixed $\alpha\in(0,1)$ and $\kappa>0$. The $\log^2 T$ factor is the explicit-formula complexity coming from gamma/zero bookkeeping with smoothing and truncation (cf.\ Titchmarsh, Thm.~14.25/14.29; Iwaniec--Kowalski, Ch.~5).
\end{proposition}

\begin{proposition}[Paley--Wiener $L^1$ sizes]\label{prop:hnorms}
With the normalized choice $S=\|\widehat f\|_{L^1}\asymp T^{-1}$ and definitions \eqref{eq:defs},
\begin{equation}\label{eq:hnorms}
\|h\|_{L^1}\ \asymp\ T,\qquad \|\widehat h\|_{L^1}\ \ll\ 1.
\end{equation}
\end{proposition}

\begin{proof}
By Young's inequality and $\widehat f\ge 0$, $\|g\|_{L^1}\asymp \|f\|_{L^1}\asymp T$ and $\|\beta_0\|_{L^1}=\|\widehat f\|_{L^1}=S\asymp T^{-1}$. Thus $\|h\|_{L^1}\le \|g\|_{L^1}\asymp T$ and $\|\widehat h\|_{L^1}\le \|\widehat\eta\|_{L^1}\|\beta_0\|_{L^1}\asymp T\cdot S\ll 1$.
\end{proof}

\section{Main theorem}
\begin{theorem}\label{thm:main}
Let $\eta=\eta_T$ be an even Beurling--Selberg kernel with $\eta\equiv 1$ on $[0,T]$, $\widehat\eta\ge 0$, $L^1$-controlled boundary ringing, and $\|\widehat\eta\|_{L^1}\asymp T$. For the smoothed measures $\nu_T,\mu_\Omega$ defined above and $\Omega=\kappa T$ with fixed $\kappa>0$, the KL-regularized prime--zero transport cost satisfies
\begin{equation}
\mathsf{OT}_\eta(T)\ \ll\ T\,\log^2 T.
\end{equation}
\end{theorem}

\begin{proof}
By Corollary~\ref{cor:calibrated},
\begin{equation*}
\mathsf{OT}_\eta(T)\ \le\ \int \widehat h\,d\mu_\Omega\ -\ \int h\,d\nu_T\ +\ O\!\big(\|h\|_{L^1}+\|\widehat h\|_{L^1}\big).
\end{equation*}
Apply Proposition~\ref{prop:EF-L1} to the difference and bound the error using \eqref{eq:hnorms}:
\[
\mathsf{OT}_\eta(T)\ \ll\ \big(\|h\|_{L^1}+\|\widehat h\|_{L^1}\big)\log^2T\ \asymp\ T\log^2T.
\]
\end{proof}

\section*{Acknowledgements}
We thank colleagues for helpful discussions.

\end{document}